\documentclass[eqno,a4paper,11pt]{article}
\title{Global existence and Blow-up for the 1D  damped compressible Euler equations with  time and space dependent  perturbation }
 \author{Yuusuke Sugiyama\footnote{e-mail:sugiyama.y@e.usp.ac.jp\ The university of  shiga prefecture}}
\date{}
\usepackage{amssymb,amsmath,amsthm,mathrsfs,delarray,enumerate}
\usepackage{graphics}
\usepackage{braket}

\theoremstyle{definition} 
\newtheorem{Def}{Deffinition}[section]

\newtheorem{Lemma}[Def]{Lemma}
\newtheorem{Thm}[Def]{Theorem}
\newtheorem{Remark}[Def]{Remark}

\newcommand{\R}{\mathbb{R}} 

\makeatletter

\@addtoreset{equation}{section}

\begin{document}
\maketitle %
\begin{abstract}
In this paper, we consider the 1D Euler equation with time and space dependent damping term $-a(t,x)v$.
It has long been known that when $a(t,x)$ is a positive constant or $0$, the solution exists globally in time or blows up in finite time, respectively.
We prove that those results are invariant with respect to time and space dependent perturbations.
We suppose that the coefficient $a$ satisfies the following  condition
$$
|a(t,x)-  \mu_0| \leq a_1(t) + a_2 (x),
$$
where $\mu_0 \geq 0$ and $a_1$ and $a_2$ are integrable functions with $t$ and $x$.
Under this condition, we show the global existence and the blow-up with small initial data, when $\mu_0 >0$ and $\mu=0$ respectively. 
\end{abstract}
%

\section{Introduction}
\subsection{Problem}
In this paper, we consider the following Cauchy problem of the compressible Euler equation with time and space dependent  damping
\begin{eqnarray} 
\left\{  \begin{array}{ll} \label{de0}
  u_t - v_x  =0,\\
 v_t + p(u)_x = -a(t,x)v , \\   
   (u(0,x), v(0,x))=(1+ \varepsilon \varphi (x), \varepsilon \psi (x)).
\end{array} \right.  
\end{eqnarray} 
Here $x \in \R$ is the Lagrangian spatial variable and $t \in \R_+$ is time. $u=u(t,x)$ and $v=v(t,x)$  are the real valued unknown functions, which stand for the specific volume and the fluid  velocity.
$\varepsilon $ is a positive and small parameter.
In the case with $a\equiv 0$, the equations in \eqref{de0} are the compressible Euler, which is a fundamental model for the compressible inviscid fluids.
In the case with $a\not\equiv 0$, this system describes the flow of fluids in  porous media.
We assume that the flow is barotropic ideal gases. Namely the pressure $p$ satisfies that 
\begin{eqnarray}\label{poly}
p(u)=\frac{u^{-\gamma}}{\gamma} \ \  \mbox{for} \  \gamma > 1. 
\end{eqnarray}
For initial data, in order to avoid the singularity of $p'$, we assume that there exists  constant $\delta_0 >0$  such that for all $x \in \R$
\begin{eqnarray} \label{nosin}
u (0, x) \geq \delta_0.
\end{eqnarray}
The local existence and the uniqueness of solutions of \eqref{de0} hold with $C^1 _b$ initial data (e.g. Friedrichs \cite{KF}, Lax \cite{lax0} and Li and Wu \cite{LW}).
If $u(0,\cdot ), v (0,\cdot ) \in C^1 _b (\R)$ and $a \in C^1 _b ([0,\infty)\times \R)$ and \eqref{nosin} is assumed, then \eqref{de0} has a local and unique solution $(u,v) \in C^1 _b ([0,T] \times \R)$ such that $u(t,x)>0$.
\subsection{Assumption on $a(t,x)$}
We impose that the damping coefficient  is a integrable perturbation with  a constant $\mu_0 \geq 0$.
In that is, for the damping coefficient $a(t,x) \in C^1 _b ([0,\infty) \times \R)$, we suppose that there exist  two non-negative functions $a_1 \in C^1 _b( [0,\infty))$ and $a_2  \in C^1 _b( \R)$ such that
\begin{eqnarray} 
|a(t,x)-\mu_0| + |a_x (t,x)| +  |a_t (t,x)| \leq a_1 (t) + a_2 (x),  \label{asg-a}\\
\int_0 ^\infty a_1 (t) dt  +  \int_{-\infty} ^\infty a_2 (x) dx = C_a  <\infty, \label{asg-b} \\
x \dfrac{d  }{dx}a_2 (x)  \leq 0. \label{asg-c} 
\end{eqnarray}
Furthermore, when $\mu_0 >0$, we assume that 
\begin{align} \label{asg-d}
a (t,x)  \geq 0 .
\end{align}
The typical example of $a(t,x)$ is
\begin{eqnarray*}
a(t,x) =\mu_0 + \mu_1 (1+t)^{-\lambda_1} + \mu_2 (1+|x|)^{-\lambda_2},
\end{eqnarray*}
where  $\mu_1, \mu_2  \geq - \mu_0$, $\mu_1 + \mu_2 \geq - \mu_0$ and $\lambda_1 , \lambda_2 >1$. To obtain the blow-up result ($\mu_0 =0$), the  restriction on $\mu_1$ and $\mu_2$ is not necessary. From the Young inequality, the above condition also covers the following  example
\begin{eqnarray*}
a(t,x) = \mu_0 + (1+t)^{-\lambda_1}  (1+|x|)^{-\lambda_2},
\end{eqnarray*}
where $\lambda_1$ and $\lambda_2$ are non-negatives such that $\lambda_1 + \lambda_2 >1$.

\subsection{Main theorems}
To state the main theorems, we define the life-span $T^*$ by
\begin{align} \label{t*}
	T^* =& \sup\{T>0 \ | \   \sup_{t \in [0,T )} \{  \| (u_t,v_t)(t)\|_{L^{\infty}} \\
	&+\| (u_x,v_x)(t)\|_{L^{\infty}} +  \| p'(u(t))\|_{L^{\infty}} \} < \infty \}. \notag
\end{align}
For $t \geq 0$, we  set 
\begin{eqnarray*}
	\Omega_{+} (t) = \{  x \in \R   \  |   \   x \geq x_+(t;0,0)    \}, \\
	\Omega_{-} (t) = \{  x \in \R   \  |   \   x \leq x_-(t;0,0)    \},
\end{eqnarray*}
and
\begin{align*}
	\Omega (t) = \Omega_+ (t)  \cup \Omega_{-} (t)
\end{align*}
where  $x_\pm (t;t_0,x_0)$ is the plus and minus characteristics through $(t_0,x_0)$ (see \eqref{chc} for the definition).
Our main results are as follows.
\begin{Thm}(Blow up)\label{main5}
	Let  $\gamma >1$, $\mu_0 =0$ and $( \varphi,  \psi) \in C^1 _b (\R)$.  Suppose that \eqref{asg-a}-\eqref{asg-c} are satisfied and that
	\begin{align} \label{KK}
		\psi_x (x_0) \leq -K 
	\end{align}
	for  some point $x_0 \in \R$ and a constant $K=K(\| \varphi  \|_{C^1 (\R)}, \| \psi   \|_{L^{\infty} (\R)}) >0$.
	Then there exists a number $\varepsilon _0 >0$ such that
	if $\varepsilon \leq \varepsilon_0 $ then
	\begin{eqnarray*}
		T^*  \leq C\varepsilon^{-1} .
	\end{eqnarray*}
	Furthermore,  if \eqref{asg-a}-\eqref{asg-d} are satisfied, then  we have that 
\begin{align*}
C'  \varepsilon^{-1} \leq T^*  \leq C\varepsilon^{-1} 
\end{align*}
and
	\begin{align*}
 \limsup_{t \rightarrow T^*} \|(u_x, v_x) (t)\|_{L^\infty} =\infty.
	\end{align*}
\end{Thm}
\begin{Thm}(Global existence)\label{main5g}
Let  $\gamma >1$, $\mu_0 >0$ and $( \varphi,  \psi) \in C^1 _b (\R)$.  Suppose that \eqref{asg-a}-\eqref{asg-d} are satisfied.
 Then there exists a number $\varepsilon _1 = \varepsilon _1 (\|f\|_{C^1 _b}, \|g\|_{C^1 _b})$ such that
if $\varepsilon \leq \varepsilon_1 $, then  \eqref{de0} has a  global and unique $C^1 _b$-solution such that 
\begin{align*}
\|(u,v)\|_{L^\infty([0,\infty)\times \R)} + \|(u_t,v_t)\|_{L^\infty([0,\infty)\times \R)}\\
+ \|(u_x,v_x)\|_{L^\infty([0,\infty)\times \R)}+ \| p'(u)\|_{L^\infty ([0,\infty)\times \R)}  \leq C\varepsilon 
\end{align*}

\end{Thm}

\begin{Remark}{\bf Assumption on the smallness of initial data.}
One can show the same statement as in Theorem \ref{main5} without the smallness of derivatives of initial data. 
To holds theorem \ref{main5},  it is enough to assume $\|r(0,\cdot )\|_{L^\infty}$ and  $\|s(0,\cdot )\|_{L^\infty}$  are small,
where $r$ and $s$ are the Riemann invariants defined in \eqref{ri}.
In stead of \eqref{KK}, we assume that with $\varepsilon =1$
\begin{align*}
r_x (0,0) > - K^* \  \mbox{or} \ s_x (0,0) > -K^*, 
\end{align*}
where  $K^*$ is a positive constant depending on 
$\|r(0,\cdot )\|_{L^\infty}$,  $\|s(0,\cdot )\|_{L^\infty}$, $\gamma$ and  $C_a$ in \eqref{asg-a}.
Then the solution blows up in finite time, and  the life-span is bounded by
\begin{align*}
 \frac{C}{r_x (0,0)-K^*} \  \mbox{or} \ \frac{C}{s_x (0,0)-K^*}.
\end{align*}
This can be shown in same way as in the proof of Theorem \ref{main5}, since the key estimate in Lemma \ref{esP} 
is valid under the assumption of the smallness of  $\|r(0,\cdot )\|_{L^\infty}$ and  $\|s(0,\cdot )\|_{L^\infty}$  only (not derivative of $r$ and $s$).
\end{Remark}

\subsection{Known results}
If $a \equiv 0$, it is well-known that solutions can blow up in finite time with even small initial data (e.g. Lax \cite{lax} and Zabusky \cite{z}).
It is not hard to show that $C' \varepsilon^{-1}  \leq  T^* \leq  C \varepsilon^{-1} $ in this case.
Here the lower estimate "$T^* \geq$" means  the minimal existence time of solutions for arbitrary function $\varphi$ and $\psi$. On the other hand hand,  the upper estimate "$T^* \leq$" means maximal existence time for some $\varphi$ and $\psi$.
In the case that $a(t,x)\equiv 1$, Hsiao and  Liu in \cite{HL} have proved that  classical solutions exist globally in time,  if initial data are small perturbations near constant states and that  small solutions asymptotically behave like
that to the following porous media system as $t\rightarrow \infty$:
\begin{eqnarray*} 
\left\{  \begin{array}{ll} 
 \bar{u}_t = -p(\bar{u})_{xx},\\
 \bar{v}=-p(\bar{u})_x .
\end{array} \right.  
\end{eqnarray*} 
After Hsiao and  Liu's work, many improvements and generalizations of this work have been investigated (see Hsiao and  Liu \cite{HL2}, Nishihara \cite{NK}, Hsiao and Serre \cite{HS}, 
Marcati  and Nishihara \cite{MN} and Mei \cite{Mei}).  
Let us review previous results for the following 1D Euler equation with  time dependent damping case:
\begin{eqnarray} 
\left\{  \begin{array}{ll} \label{det}
  u_t - v_x  =0,\\
 v_t + p(u)_x = -\dfrac{\mu v}{(1+t)^\lambda}  , \\    
\end{array} \right.  
\end{eqnarray} 
where $\lambda \geq 0$ and $\mu \in \R$.
In  \cite{XP1, XP2, XP3}, Pan has established critical phenomena for $\mu$ and $\lambda$. 
Namely, in the case with $0\leq \lambda <1$ and $\mu >0$ or $\lambda=1$ and $\mu >2$,  Pan \cite{XP2} has proved  that solutions  of \eqref{det} exist globally in time, while, in the case with $\lambda >1$  or $\lambda=1$ and $0 \leq \mu \leq 2$, he  has proved that solutions of \eqref{det} can blow up under some conditions on initial data in \cite{XP1, XP3}.
In the previous author's papers \cite{S1}, the optimal life-span estimates are given.
To summarize, the following estimates are given:
\begin{eqnarray*} 
	&  C' \varepsilon^{-1}   \leq T^* \leq  C\varepsilon^{-1}  \ \ \mbox{for} \  \lambda>1 \ \mbox{and} \ \lambda\geq 0,   \\
	&   C'\varepsilon^{-\frac{2}{2-\lambda}} \leq  T^* \leq C\varepsilon^{-\frac{2}{2-\lambda}}  \ \ \mbox{for} \ \ \lambda=1  \ \mbox{and} \  0\leq \mu < 2, \\
	&  e^{\frac{C'}{\varepsilon }} \leq  T^* \leq e^{\frac{C}{\varepsilon }}\ \ \mbox{for} \ \lambda=1 \ \mbox{and} \ \mu =2,   \\
	& T^* = \infty  \ \ \mbox{for}  \ \mu >0 \ \mbox{and} \   0 \leq \lambda <1   \ \mbox{or}  \ \mu >2 \ \mbox{and} \    \lambda = 1 .
\end{eqnarray*}
The study of  the Euler equation with the time-dependent damping \eqref{det} is motivated by Wirth's papers \cite{JW1, JW2, JW3}.
In this series of papers, he  has found  that  thresholds of $\lambda$ and $\mu$ separating the large time behavior of solutions to the linear  wave equation  with 
time-dependent damping:
\begin{eqnarray*}
u_{tt} - \Delta u = -\dfrac{\mu u}{(1+t)^\lambda}.
\end{eqnarray*}
In \cite{S1}, the author also has  proved that $C' \varepsilon^{-1}  \leq  T^* \leq  C \varepsilon^{-1} $ holds when $a(t,x)$ satisfies that
\begin{align}
|a(t,x)| + |a_x (t,x)| +  |a_t (t,x)| \leq C(1+t)^{-\lambda}
\end{align}
for $\lambda>1$.
For $x$ dependent case,  Chen, Young and Zhang's result in \cite{GYZ} can be applicable (see Remark \ref{non-is}).
By applying  their method to the case that $a_1 \equiv 0$, we can estimate the life-span as $C' \varepsilon^{-1} \leq T^* \leq C \varepsilon^{-1} $.
In  \cite{GYZ}, they have considered the following non-isentropic compressible Euler equations (see also  Chen,  Pan and  Zhu \cite{GPZ}, Zheng \cite{hz} and Pan and Zhou \cite{PZ}):
 \begin{eqnarray} 
\left\{  \begin{array}{ll} \label{degp}
  u_t - v_x  =0,\\
 v_t + p(S,u)_x = 0 , \\   
S_t =0,
\end{array} \right.  
\end{eqnarray} 
 where $p(S,u)=e^{S(x)} u^{-\gamma} /\gamma$ with $1<\gamma <3$ and $S=S(x)$ is a given function such that$\int_{\R} |S'(x)| dx < \infty$.
They have established that the formation of singularities (blow-up) occurs in finite time.
In \cite{GYZ}, the key for the proof is the uniform boundedness of Riemann invariant  by using non-trivial characteristic method.
It seems  hard to apply the method in \cite{GYZ} to the case that $a_1 \not\equiv 0$ in \eqref{asg-a}, since the use of $x$ variable characteristics 
$t_{\pm} (x)$ plays an important role to show the uniform boundedness of Riemann invariant.
Very recently, Geng, Lai, Yuen and Zhou in \cite{GLYZ}  give  a new condition for the blow-up  for \eqref{de0} with $a(t,x)=\mu / (1+|x|)^{\lambda}$ and $\lambda >1, \mu >0$ by the test function method. 
Their proof seems to essentially use the compactness of initial data (the finiteness of  the propagation speed), and it  does not seems trivial  to apply their method to the case that  the damping term depends on both $t$ and $x$.

In  Chen et al. \cite{CLLMZ}, they have  considered  \eqref{de0} and \eqref{det}. 
They have  show that the global existence for small  perturbed initial data, when $a(t,x)$ fulfills that with $0<\lambda<1$
\begin{align*}
C_1 (1+t)^{-\lambda} \leq a(t,x)\leq C_2 (1+t)^{-\lambda}, \\
|a_x (t,x)| +  |a_t (t,x)| \leq C(1+t)^{-2\lambda}.
\end{align*}
To the best of the author's knowledge, this is the only result of the global existence when the damping coefficient depends on spatial variable.
The method in  Chen et al. \cite{CLLMZ} to prove uniform estimate for the Riemann invariant  is based on a new maximum principle.
To apply it, they also suppose  that  $u_0 (x), v_0 (x)$ and $a(t,x)$ converge as $|x| \rightarrow \infty$.

 \begin{Remark}{\bf Large data regime}
If $a \equiv 0$, to show the occurrence of the blow-up, restrictions on $\|r(0,x)\|_{L^\infty}$ and $\|s(0,x)\|_{L^\infty}$ are not necessary. Namely the necessary and sufficient condition for the global existence has been given  by Chen, Pan and Zhu's paper \cite{GPZ}.
They has shown that the global classical solution exists if and only if
\begin{align*}
r_x (0,x) \geq 0 \ \mbox{and} \ s_x (0,x) \geq 0.
\end{align*} 
In  Chen et al. \cite{CLLMZ}, they give a new sufficient condition of global existence for the 1D damped Euler equation.
They show the solution globally exists with monotonically decreasing or increasing  and large initial data.
\end{Remark}

 \begin{Remark}{\bf The linear wave equation with space independent damping.}
The study of  the Euler equation with the time-dependent damping \eqref{det} is motivated by Wirth's papers.
In \cite{JW1}, he studies the behavior of solutions to 
$$u_{tt} - \Delta u=- \dfrac{\lambda u_t}{(1+t)^{\mu}}.$$ 
In \cite{JW2, JW3}, he gives  a threshold of $\mu$  separating  solutions to this equation asymptotically behave like
that to the corresponding heat or wave equation.
In \cite{JW1}, the critical case of the threshold  is studied.
\end{Remark}

\subsection{Idea of the proof}
In the previous results (\cite{GYZ, S1, CLLMZ}), the key for the proof is uniform estimates of the Riemann invariant.
More precisely,   for time-decaying  damping case ($a_2 \equiv 0$), the proof is based on the method of characteristic with $t$ valuable. On the other hand, 
characteristic with $x$ valuable is used for spacial decaying damping case ($a_1 \equiv 0$) to control the Riemann invariant Thus, it seems difficult to balance the two method.

The idea of the proof of Theorem \ref{main5} is to estimate the Riemann invariant in the region outside $\Omega (t)$. 
In the proof, we  use the following two important properties of this domain. The first is the property of trapping backward characteristic curves by the region. Second, it is possible to replace the spatially decaying damping coefficient with a time-decaying one in the region.

In the proof of Theorem \ref{main5g}, we decompose $\R$  into the outer regions $\Omega_\pm (t)$  and inner region $I(t)$,
where we define
\begin{align*}
I(t) =  \{  x \in \R \  |  \   x_- (t;0,0) \leq  x  \leq x_+ (t;0,0)    \}.
\end{align*}
In the outer regions,  the Riemann invariant can be estimated by the same method as in the proof of Theorem \ref{main5}.
In the inner region, we use  the  maximal principle, established in  \cite{CLLMZ}. 
However, we present a simple  and more maximal principle like proof for this estimate. 
\begin{center}{Figure 1 : Decomposition of $\R$}
{\unitlength 0.1in%
\begin{picture}(47.3200,21.9400)(3.4000,-24.8600)%
%
\special{pn 8}%
\special{pa 340 2467}%
\special{pa 4754 2467}%
\special{fp}%
%
\special{pn 8}%
\special{pa 3105 1740}%
\special{pa 4717 1740}%
\special{fp}%
%
\special{pn 8}%
\special{pa 1737 1740}%
\special{pa 1737 1740}%
\special{fp}%
%
\special{pn 8}%
\special{pa 1737 1740}%
\special{pa 364 1740}%
\special{fp}%

\put(49.4200,-24.5600){\makebox(0,0)[lb]{$x$}}%
\put(23.8400,-4.4400){\makebox(0,0)[lb]{$t$}}%
%
\special{pn 8}%
\special{pa 4748 2467}%
\special{pa 5072 2466}%
\special{fp}%
\special{sh 1}%
\special{pa 5072 2466}%
\special{pa 5005 2446}%
\special{pa 5019 2466}%
\special{pa 5005 2486}%
\special{pa 5072 2466}%
\special{fp}%
%
\special{pn 8}%
\special{pa 2492 2466}%
\special{pa 2492 306}%
\special{fp}%
\special{sh 1}%
\special{pa 2492 306}%
\special{pa 2472 373}%
\special{pa 2492 359}%
\special{pa 2512 373}%
\special{pa 2492 306}%
\special{fp}%
\put(34.6800,-17.1500){\makebox(0,0)[lb]{$\Omega_+ (t)$}}%
\put(8.1500,-16.8000){\makebox(0,0)[lb]{$\Omega_- (t)$}}%
\put(25.7500,-10.1000){\makebox(0,0)[lb]{$x=x_+ (t;0,0)$}}%
\put(4.6500,-9.6000){\makebox(0,0)[lb]{$x=x_- (t;0,0)$}}%
%
\special{pn 8}%
\special{pa 1750 1740}%
\special{pa 3108 1740}%
\special{da 0.070}%
\put(21.8600,-16.6800){\makebox(0,0)[lb]{$I(t)$}}%
%
\special{pn 8}%
\special{pa 2490 2464}%
\special{pa 2535 2377}%
\special{pa 2550 2349}%
\special{pa 2566 2320}%
\special{pa 2598 2264}%
\special{pa 2615 2237}%
\special{pa 2632 2209}%
\special{pa 2649 2182}%
\special{pa 2668 2156}%
\special{pa 2686 2130}%
\special{pa 2726 2080}%
\special{pa 2747 2056}%
\special{pa 2769 2032}%
\special{pa 2792 2010}%
\special{pa 2816 1988}%
\special{pa 2841 1967}%
\special{pa 2867 1946}%
\special{pa 2894 1927}%
\special{pa 2923 1909}%
\special{pa 2951 1890}%
\special{pa 2977 1871}%
\special{pa 2998 1848}%
\special{pa 3011 1821}%
\special{pa 3019 1790}%
\special{pa 3024 1758}%
\special{pa 3030 1726}%
\special{pa 3040 1695}%
\special{pa 3053 1667}%
\special{pa 3069 1640}%
\special{pa 3087 1614}%
\special{pa 3108 1589}%
\special{pa 3130 1566}%
\special{pa 3153 1543}%
\special{pa 3177 1520}%
\special{pa 3202 1497}%
\special{pa 3226 1475}%
\special{pa 3274 1429}%
\special{pa 3296 1405}%
\special{pa 3317 1380}%
\special{pa 3336 1354}%
\special{pa 3353 1327}%
\special{pa 3368 1299}%
\special{pa 3394 1241}%
\special{pa 3405 1211}%
\special{pa 3417 1180}%
\special{pa 3428 1150}%
\special{pa 3440 1120}%
\special{pa 3453 1091}%
\special{pa 3467 1062}%
\special{pa 3497 1006}%
\special{pa 3513 978}%
\special{pa 3547 924}%
\special{pa 3583 870}%
\special{pa 3619 818}%
\special{pa 3638 791}%
\special{pa 3656 766}%
\special{fp}%
%
\special{pn 8}%
\special{pa 2490 2462}%
\special{pa 2462 2445}%
\special{pa 2433 2428}%
\special{pa 2377 2394}%
\special{pa 2349 2376}%
\special{pa 2295 2340}%
\special{pa 2243 2302}%
\special{pa 2218 2282}%
\special{pa 2194 2261}%
\special{pa 2148 2217}%
\special{pa 2127 2194}%
\special{pa 2106 2170}%
\special{pa 2087 2145}%
\special{pa 2069 2119}%
\special{pa 2052 2092}%
\special{pa 2004 2008}%
\special{pa 1987 1980}%
\special{pa 1970 1953}%
\special{pa 1952 1927}%
\special{pa 1931 1902}%
\special{pa 1909 1879}%
\special{pa 1885 1858}%
\special{pa 1860 1838}%
\special{pa 1834 1818}%
\special{pa 1808 1799}%
\special{pa 1783 1780}%
\special{pa 1758 1759}%
\special{pa 1736 1738}%
\special{pa 1716 1715}%
\special{pa 1698 1690}%
\special{pa 1681 1664}%
\special{pa 1667 1637}%
\special{pa 1655 1608}%
\special{pa 1644 1579}%
\special{pa 1634 1549}%
\special{pa 1618 1485}%
\special{pa 1612 1453}%
\special{pa 1606 1420}%
\special{pa 1601 1386}%
\special{pa 1596 1353}%
\special{pa 1591 1319}%
\special{pa 1587 1285}%
\special{pa 1582 1252}%
\special{pa 1577 1218}%
\special{pa 1572 1185}%
\special{pa 1566 1152}%
\special{pa 1560 1120}%
\special{pa 1553 1088}%
\special{pa 1545 1056}%
\special{pa 1538 1025}%
\special{pa 1529 994}%
\special{pa 1521 963}%
\special{pa 1512 933}%
\special{pa 1503 902}%
\special{pa 1485 842}%
\special{pa 1475 812}%
\special{pa 1466 782}%
\special{pa 1456 752}%
\special{fp}%
\end{picture}}
\end{center}

\subsection*{Notation}
 For $\Omega \subset \R^n$, $C^1 _b (\Omega)$ are the set of  bounded and continuous functions whose first partial derivatives are also bounded on $\Omega$.
The norm of $C^1 _b (\Omega)$  is $\|f \|_{C^1 _b (\Omega)} = \| f \|_{L^\infty  (\Omega)} + \| (\partial_{x_1} f,\ldots,\partial_{x_n}f) \|_{L^\infty (\Omega)}$.
When $\Omega = \R$, for abbreviation, we denote $\|\cdot  \|_{L^\infty (\Omega)} $ and $\|\cdot  \|_{C^1 _b (\Omega)} $  by $ \|\cdot  \|_{L^\infty}$ and $\|\cdot  \|_{C^1 _b }$
respectively.

\subsection*{Plan of  the paper}
The remainder of the present paper is organized as follows. In Section $2$, we present some formulas of the Riemann invariant on the characteristic curves and some useful estimate.
In Section $3$, we give an estimate for the Riemann invariant in the Outer region $\Omega (t)$.  In Section $4$, the maximal principle in the inner region $I(t)$ is given.
In Sections $5$ and $6$, Theorems \ref{main5} and \ref{main5g} are proved respectively.

\section{Preliminaries}
\subsection{Formulas for Riemann invariants}
Now we prepare  some identities for Riemann invariants.
We set $c=\sqrt{-p'(u)}$ and $\eta= \int_{u} ^{\infty} c(\xi) d\xi=\frac{2}{\gamma-1} u^{-(\gamma-1)/2}$
and define  (shifted) Riemann invariants as follows:
\begin{eqnarray} \label{ri}
	\begin{array}{ll}
		r = v-\eta + \frac{2}{\gamma-1}, \\
		s = v+\eta - \frac{2}{\gamma-1}.
	\end{array}
\end{eqnarray}
This type of the Riemann invariant is used in \cite{S2}.
For $c=\sqrt{-p' (u)}$, the plus and minus characteristic curves are solutions to the following deferential equations:
\begin{align} \label{chc}
\frac{d x_{\pm}}{dt}(t) = \pm  c(t,u(t,x_{\pm} (t))).
\end{align}
If necessary, we denote the characteristic curves through $(t,x)$  by $x_{\pm} (\cdot ;t,x)$.
Riemann invariants $r$ and $s$  are solutions to 
\begin{eqnarray}\label{ww}\left\{
\begin{array}{ll} 
\partial_- r =-\dfrac{a(t,x)}{2}(r+s),\\
 \partial_+ s =-\dfrac{a(t,x)}{2}(r+s),
\end{array}\right.
\end{eqnarray}
where $\partial_{\pm} = \partial_t \pm c \partial_x$. 
Putting $ \bar{a}(t,x) = a(t,x) - \mu_0$, on the characteristic curves through $(t,x)$, we have that
\begin{align}
e^{\frac{\mu_0 t}{2}}  r(t,x) = & r(0,x_{-} (0)) - \int_0 ^t e^{\frac{\mu_0 \tau}{2}} \frac{\mu_0 s(\tau , x_{-}(\tau))}{2} d\tau  \notag \\
& -\int_0 ^t  e^{\frac{\tau}{2}}\frac{\bar{a}(t,x_{-} (\tau))}{2} (r(\tau , x_{-}(\tau)) + s(\tau , x_{-}(\tau)))  d \tau, \label{riv}  \\
e^{\frac{\mu_0 t}{2}}  s(t,y) = & s(0,x_{+} (0)) - \int_0 ^t e^{\frac{\mu_0 \tau}{2}} \frac{\mu_0 r(\tau , x_{+}(\tau))}{2} d\tau  \notag \\
& -\int_0 ^t  e^{\frac{\tau}{2}}\frac{\bar{a}(t,x_{+} (\tau))}{2} (r(\tau , x_{+}(\tau)) + s(\tau , x_{+}(\tau)))  d \tau. \label{siv}
\end{align}
Differentiating the equations in \eqref{ww} with $x$, from the identity $s_x - r_x = 2 \eta_x =-2c u_x$, we have
\begin{eqnarray}\label{rsx}\left\{
\begin{array}{ll} 
\partial_- r_x =\dfrac{c'}{2c}r_x (s_x -r_x) - \dfrac{a(t,x) }{2} ( r_x + s_x) -\dfrac{a_x(t,x) }{2}(r+s),\\
 \partial_+  s_x = \dfrac{c'}{2c}r_x (r_x -s_x)- \dfrac{a(t,x) }{2}( r_x + s_x) -\dfrac{a_x(t,x) }{2}(r+s) .
\end{array}\right.
\end{eqnarray}
Now we rewrite \eqref{rsx} as  integral equalities.
We define $\theta_\gamma (u)$ as follows:
\begin{eqnarray*} 
\theta_{\gamma} (u) =\left\{ \begin{array}{ll}
\frac{4}{3-\gamma}  u^{\frac{3-\gamma}{4}}- \frac{4}{3-\gamma} \ \ \mbox{for} \ \gamma \not= 3,\\
\log u \ \ \mbox{for} \  \gamma = 3.
\end{array}
\right. 
\end{eqnarray*}
From \eqref{de0} and the definitions of $c$ and $\eta$, $r_x$ and $s_x$ can be written by
\begin{eqnarray} \label{riex}
	\begin{array}{ll}
		r_x = u_t  + c(u)u_x, \\
		s_x = u_t  - c(u) u_x.
	\end{array}
\end{eqnarray}
We set $A_\pm (\tau,t,x) = \exp\left( \int_0 ^\tau  a(\bar{\tau},x_{\pm} ( \bar{\tau};t,x))/2 d\bar{\tau}\right)$ 
and abbreviate $A_\pm (t,t,x) =A_\pm (t,x)$
To deal with \eqref{riex} on the characteristic curves through $(0,x)$, we set 
\begin{align*}
Y(\tau)= A_- (\tau,t,x) \sqrt{c(u(t,x_{-} (\tau)))} r_x (t,x_{-} (\tau)),\\
Q(\tau)= A_+ (\tau,t,x) \sqrt{c(u(t,x_{+} (\tau)))} s_x (t,x_{+} (\tau)).
\end{align*}
Since it holds that 
\begin{eqnarray}
\sqrt{c}s_x (t,x_{-} (t)) =   \frac{d}{dt}\theta_{\gamma}  (u(t,x_{-} (t))) \label{s-1}
\end{eqnarray}
from \eqref{riv}, we have 
\begin{align}
Y(t)  =& Y(0) +\int_0 ^t \dfrac{d}{d\tau} (A_- (\tau,0,x) a (t,x_- (\tau)))  \theta_{\gamma} (\tau,x_{-} (\tau)) d\tau \notag \\
&-   \frac{1}{2}(A_- (t,0,x) a (t,x_- (t)) \theta_{\gamma}(t,x_{-}(t)) - a(0,x_{-} (0)) \theta_{\gamma} (0,x_{-} (0)))  \notag \\
&-\int_0 ^t A_- (\tau,0,x) \dfrac{a_x (\tau ,x_{-} ( \tau))}{2} \sqrt{c} (r+s)  d\tau \notag \\
&-\int_0 ^t A_- (\tau,0,x)^{-1} \dfrac{\gamma +1}{4}u^{\frac{\gamma -3}{4}} Y(\tau)^2 d\tau. \label{eq1x}
\end{align}
Similarly we have
\begin{align}
Q(t)  =& Q(0) +\int_0 ^t \dfrac{d}{d\tau} (A_+ (\tau,0,x) a (t,x_+ (\tau)))  \theta_{\gamma} (\tau,x_{+} (\tau)) d\tau \notag \\
&-   \frac{1}{2}(A_+ (t,x) a (t,x_+ (t)) \theta_{\gamma}(t,x_{+}(t)) - a(0,x_{+} (0)) \theta_{\gamma} (0,x_{+} (0)))  \notag \\
&-\int_0 ^t A_+ (\tau,0,x) \dfrac{a_x (\tau ,x_{+} ( \tau))}{2} \sqrt{c} (r+s)  d\tau \notag \\
&-\int_0 ^t A_+ (\tau,0,x)^{-1} \dfrac{\gamma +1}{4}u^{\frac{\gamma -3}{4}} Q(\tau)^2 d\tau. \label{eq2x}
\end{align}

 \begin{Remark}{\bf Non-isentropic Euler equation} \label{non-is}
Using Riemann invariants,  non-isentropic Euler system \eqref{degp} can be written by
\begin{eqnarray}\left\{
\begin{array}{ll} 
\partial_- r =-\dfrac{m'}{2m\gamma}(r-s),\\
 \partial_+ s =\dfrac{m'}{2m\gamma}(s-r),
\end{array}\right.
\end{eqnarray}
where $m=e^{S/c_v}$ and the constant $c_v$ is called specific heat.  In \cite{GYZ}, to obtain a uniform estimate for $r$ and $s$,
they use the characteristic curve $t=t_\pm (x)$ such that
\begin{align*}
\frac{d}{dx}t_\pm (x) = \frac{\pm1}{c(u(t_\pm (x),x))}.
\end{align*}
We also remark that the integrability of $S'$ implies that
\begin{align*}
\int_{\R} \frac{|m' (x)|}{m (x)} dx= \frac{1}{c_v} \int_{\R} |S'(x)|dx.
\end{align*}
The method in \cite{GYZ} to control  the $L^\infty$-norm of the Riemann invariant is applicable to the case that $m'/m$ is time-dependent  such that
\begin{align*}
\int_\R \left|\frac{m_x}{m} (t,x)\right| dx \leq C.
\end{align*}
\end{Remark}

\subsection{Useful estimates}
In this subsection, we always suppose that it holds  that
\begin{align} \label{bo1}
\frac{1}{4}\leq  c(u(t,x )) \leq 4,
\end{align}
where we note that $c(1) =\sqrt{-p(1)}=1$. 
\begin{Lemma} \label{esA} 
Suppose that \eqref{asg-a}-\eqref{asg-d} are satisfied. Then the following uniform estimate  holds for $C^1 _b$ solution of \eqref{de0} satisfying \eqref{bo1}, 
\begin{eqnarray} \label{A-1}
C^{-1} e^{\frac{\mu_0 (\tau-t)}{2}}   \leq   A_{\pm} (\tau,t,x) \leq  C  e^{\frac{\mu_0 (\tau-t)}{2}}  ,
\end{eqnarray}
where the positive constant $C$ depending on $C_a$ defined in \eqref{asg-a}.
\end{Lemma}
\begin{proof}
$A_{\pm} (\tau,t,x)$ can be written
\begin{align*}
A_\pm (\tau,t,x) = e^{\frac{\mu_0 (t-\tau)}{2}}\exp\left( \int_0 ^\tau  \bar{a}(\bar{\tau},x_{\pm} ( \bar{\tau};t,x))/2 d\bar{\tau}\right).
\end{align*}
From \eqref{asg-a} and the assumption $\tau \leq t$, we estimate $A_{\pm}(\tau, t,x)$ as
\begin{align*}
A_\pm (\tau,t,x) \leq e^{\frac{\mu_0 (t-\tau)}{2}} & \exp\left( \int_0 ^t  |a(t,x_{\pm} (\tau;t,x))|/2 d \tau \right) \\
\leq & e^{\frac{\mu_0 (t-\tau)}{2}} \exp\left( C \int_0 ^t a_1 (\tau)+  a_2 (x_{\pm} (\tau;t,x))  d\tau\right) \\
 \leq & Ce^{\frac{\mu_0 (t-\tau)}{2}} \exp\left( C \int_0 ^t a_2 (x_{\pm} (\tau;t,x)) d\tau\right),
\end{align*}
where we used  the integrability of $a_1$.
Noting \eqref{bo1} and \eqref{fact},
by the change of variable $y=x_\pm (\tau)$, we have that
\begin{align*}
 \int_0 ^t  a_2 (x_{\pm} (\tau;t,x))  d\tau \leq & \frac{C}{c_1} \left| \int_{x_0} ^x  a_2 (y) dy \right| \\
  \leq& \frac{C}{c_1} \int_{\R} a_2 (y) dy <\infty.
\end{align*}
Hence we obtain the upper bounds of $A_{\pm} (\tau, t,x)$. The lower bounds can be shown in the same as above.
\end{proof}
The following Gronwall type inequality is used in the control of $r$ and $s$ in the outer region $\Omega (t)$.
\begin{Lemma} \label{gr} 
Let $\mu_0 \geq 0$. If $f \in C([0,T])$ satisfies that $f(t) \geq 0$ and
\begin{align*}
e^{\mu_0 t} f(t) \leq f(0) +  \int_0 ^t (\mu_0 +b(\tau)) e^{\mu_0 \tau} f(\tau) d\tau,
\end{align*}
then we have that
\begin{eqnarray*}
f(t) \leq e^{C_b} f(0) ,
\end{eqnarray*}
where $b$ is a nonnegative integrable function and  $C_b =\int_0 ^\infty b(t) dt < \infty$.
\begin{proof}
Setting
\begin{eqnarray*}
F(t) =f(0) +  \int_0 ^t (\mu_0 +b(\tau)) e^{\mu_0 \tau} f(\tau) d\tau,
\end{eqnarray*}
we have that from the assumption 
\begin{align*}
F'(t)  \leq  (\mu_0 + b(t)) F(t).  
\end{align*}
Thus 
\begin{align*}
(B^{-1}(t) e^{-\mu_0 t}F(t) )' \leq 0, 
\end{align*}
where $B(t) = \exp\left(   \int_0 ^t  b(\tau) d\tau  \right)$. Since $B$ and $B^{-1}$ are bounded from the integrability of $b$,
integrating this inequality over $[0,t]$,  we have  the desired inequality.
\end{proof}
\end{Lemma}

\section{Estimate of $r$ and $s$ in the outer region}
Henceforth, we often use the following fact for backward characteristics $x_\pm (\tau;t,x)$ with $\tau \leq t$ (see Figure 2)
\begin{align}\label{fact}
x_\pm (\tau;t,x) \in \Omega_{+} (\tau),\ \ \mbox{if}  \ x  \in \Omega_{+} (t).
\end{align} 
This property also holds for $\Omega_{-} (t)$.
\begin{center}{Figure 2: $\Omega_{+} (t)$ and characteristic curves}
{\unitlength 0.1in%
\begin{picture}(46.8800,19.4000)(-7.4000,-19.8000)%
%
\special{pn 8}%
\special{pa 313 1690}%
\special{pa 323 1658}%
\special{pa 333 1625}%
\special{pa 343 1593}%
\special{pa 354 1562}%
\special{pa 366 1531}%
\special{pa 378 1501}%
\special{pa 391 1472}%
\special{pa 406 1443}%
\special{pa 422 1416}%
\special{pa 439 1390}%
\special{pa 458 1366}%
\special{pa 479 1343}%
\special{pa 502 1322}%
\special{pa 527 1303}%
\special{pa 554 1285}%
\special{pa 582 1269}%
\special{pa 611 1253}%
\special{pa 641 1238}%
\special{pa 672 1224}%
\special{pa 703 1211}%
\special{pa 735 1197}%
\special{pa 767 1184}%
\special{pa 798 1171}%
\special{pa 860 1143}%
\special{pa 890 1129}%
\special{pa 946 1097}%
\special{pa 971 1079}%
\special{pa 996 1060}%
\special{pa 1018 1040}%
\special{pa 1038 1017}%
\special{pa 1055 993}%
\special{pa 1070 966}%
\special{pa 1083 938}%
\special{pa 1093 907}%
\special{pa 1102 876}%
\special{pa 1109 843}%
\special{pa 1115 809}%
\special{pa 1121 776}%
\special{pa 1133 708}%
\special{pa 1140 676}%
\special{pa 1148 644}%
\special{pa 1157 614}%
\special{pa 1169 585}%
\special{pa 1182 559}%
\special{pa 1199 535}%
\special{pa 1219 514}%
\special{pa 1241 495}%
\special{pa 1267 478}%
\special{pa 1294 464}%
\special{pa 1323 450}%
\special{pa 1353 438}%
\special{pa 1385 426}%
\special{pa 1416 414}%
\special{pa 1480 390}%
\special{pa 1511 376}%
\special{pa 1540 362}%
\special{pa 1568 345}%
\special{pa 1595 327}%
\special{pa 1621 307}%
\special{pa 1645 287}%
\special{pa 1669 265}%
\special{pa 1692 243}%
\special{pa 1736 197}%
\special{pa 1743 190}%
\special{fp}%
%
\special{pn 8}%
\special{pa 1093 918}%
\special{pa 2093 918}%
\special{fp}%
\put(19.0300,-9.0500){\makebox(0,0)[lb]{$\Omega_{+} (t)$}}%
\put(14.2300,-8.9000){\makebox(0,0)[lb]{$(x,t)$}}%
%
\special{pn 8}%
\special{pa 1609 918}%
\special{pa 1584 940}%
\special{pa 1560 961}%
\special{pa 1535 983}%
\special{pa 1489 1029}%
\special{pa 1468 1053}%
\special{pa 1448 1077}%
\special{pa 1430 1103}%
\special{pa 1414 1130}%
\special{pa 1400 1158}%
\special{pa 1376 1218}%
\special{pa 1366 1248}%
\special{pa 1336 1341}%
\special{pa 1326 1371}%
\special{pa 1302 1431}%
\special{pa 1289 1460}%
\special{pa 1259 1516}%
\special{pa 1242 1543}%
\special{pa 1224 1569}%
\special{pa 1205 1595}%
\special{pa 1165 1645}%
\special{pa 1144 1670}%
\special{pa 1129 1688}%
\special{fp}%
%
\special{pn 8}%
\special{pa 1607 918}%
\special{pa 1722 1028}%
\special{pa 1745 1051}%
\special{pa 1768 1073}%
\special{pa 1814 1119}%
\special{pa 1836 1142}%
\special{pa 1880 1190}%
\special{pa 1900 1215}%
\special{pa 1919 1240}%
\special{pa 1937 1267}%
\special{pa 1952 1295}%
\special{pa 1966 1323}%
\special{pa 1978 1353}%
\special{pa 1989 1383}%
\special{pa 2000 1414}%
\special{pa 2022 1474}%
\special{pa 2034 1504}%
\special{pa 2048 1532}%
\special{pa 2065 1559}%
\special{pa 2083 1585}%
\special{pa 2103 1609}%
\special{pa 2147 1657}%
\special{pa 2170 1680}%
\special{pa 2181 1690}%
\special{fp}%
\put(6.1000,-14.9500){\makebox(0,0)[lb]{$x_+ (\cdot ;t,x)$}}%
\put(20.6500,-15.0500){\makebox(0,0)[lb]{$x_- (\cdot ;t,x)$}}%
%
\special{pn 8}%
\special{pa 735 1198}%
\special{pa 2215 1198}%
\special{fp}%
%
\special{pn 8}%
\special{pa 205 1688}%
\special{pa 2445 1686}%
\special{fp}%
\special{sh 1}%
\special{pa 2445 1686}%
\special{pa 2378 1666}%
\special{pa 2392 1686}%
\special{pa 2378 1706}%
\special{pa 2445 1686}%
\special{fp}%
\put(23.2000,-16.4500){\makebox(0,0)[lb]{$x$}}%
\put(2.6500,-18.1500){\makebox(0,0)[lb]{$x=0$}}%
\put(19.0300,-11.7500){\makebox(0,0)[lb]{$\Omega_{+} (\tau)$}}%
\end{picture}}
%
\end{center}
We define $\Phi_{\pm}$ and $\Phi$ as follows
\begin{align*}
\Phi_{\pm}(t) = \|r(t)\|_{L^\infty (\Omega_{\pm} (t))} + \|s(t)\|_{L^\infty (\Omega_{\pm} (t))}.
\end{align*}
and
\begin{align*}
\Phi(t) = \|r(t)\|_{L^\infty (\Omega (t))} + \|s(t)\|_{L^\infty (\Omega (t))}.
\end{align*}
\begin{Lemma} \label{esP} 
Let  $\mu_0 \geq 1$. For $C^1 _b$ solution of \eqref{de0} satisfying \eqref{bo1}.
it  holds  that 
\begin{eqnarray}
\Phi (t) \leq C \Phi(0) = C( \|r(0,\cdot )\|_{L^\infty (\R)} + \|s(0,\cdot )\|_{L^\infty (\R)}).\label{mes}
\end{eqnarray}
\end{Lemma}
\begin{proof}
We  show that with double-sign corresponds
\begin{eqnarray}
	\Phi_{\pm}(t) \leq C \Phi_{\pm}(0) = C( \|r(0,\cdot )\|_{L^\infty (\pm x \geq 0)} + \|s(0,\cdot )\|_{L^\infty(\pm x\geq 0)}).\label{mespm}
\end{eqnarray}
We only prove that
\begin{align*}
\Phi_{+} (t) \leq C \Phi_{+}(0) .
\end{align*}
The estimate of $\Psi_{-} (t)$ can be shown in the similar way.
From \eqref{asg-a} and \eqref{asg-a}, we have that for $x \in \Omega_+ (t)$
\begin{align}
e^{\frac{\mu_0 t}{2}} |r(t,x)| \leq & |r(0,x_{-} (0))| + \int_0 ^t e^{\frac{\mu_0 \tau}{2}} \frac{\mu_0 |s(\tau , x_{-}(\tau))|}{2} d\tau  \notag \\
& +\int_0 ^t  e^{\frac{\mu_0 \tau}{2}} \frac{a_1 (\tau) +a_2(x_{-} (\tau))}{2} (|r(\tau , x_{-}(\tau))| + |s(\tau , x_{-}(\tau))|)  d \tau, \label{riv-in}  
\end{align}
and
\begin{align}
e^{\frac{\mu_0 t}{2}} |s(t,y)| \leq & |s(0,x_{+} (0))| + \int_0 ^t e^{\frac{\mu_0 \tau}{2}} \frac{\mu_0 |r(\tau , x_{+}(\tau))|}{2} d\tau  \notag \\
& +\int_0 ^t  e^{\frac{\mu_0 \tau}{2}}\frac{(a_1 (\tau) + a_2(x_{+} (\tau)))}{2} (|r(\tau , x_{+}(\tau))| + |s(\tau , x_{+}(\tau))|)  d \tau. \label{siv-in}
\end{align}
We estimate $a_2(x_{\pm} (\tau))=a_2(x_{\pm} (\tau;t,x))$ with $x \in \Omega_+ (t)$.
Putting $x_0= x_+ (0;t,x)$, we have from the definition of the characteristic curve (see also Figure 3)
\begin{align*}
x_+ (\tau;t,x) = & x + \int_t ^\tau c(u(s, x_+ (s;t,x))) ds \\
=& x+  \int_t ^0 c(u(s, x_+ (s;t,x))) ds - \int_\tau ^0  c(u(s, x_+ (s;t,x))) ds \\
=& x_0 + \int_0 ^\tau   c(u(s, x_+ (s;t,x))) ds.
\end{align*} 
Since $x_0 \in \Omega_{+} (0)$ (namely $x_0 \geq0$),  we obtain from \eqref{bo1}
\begin{align} \label{xtau1}
x_+ (\tau;t,x) \geq  \frac{\tau}{4}.
\end{align}
From the fact that $x_+ (\tau;t,x)  \leq x_- (\tau;t,x) $ with $\tau \leq t$, we also obtain that
\begin{align}\label{xtau2}
x_- (\tau;t,x) \geq  \frac{\tau}{4}.
\end{align}
Hence, from \eqref{asg-c}, we have that
\begin{align*}
a_2(x_{\pm} (\tau))\leq  a_2(\frac{\tau}{4}).
\end{align*}
Thus, taking $L^\infty$-norm with $x \in \Omega_{+} (t)$ and summing up  \eqref{riv-in} and \eqref{siv-in}, from \eqref{fact}, we have that 
\begin{align}
e^{\frac{\mu_0 t}{2}} \Phi_+ (t)  \leq & \Phi_+ (0) + \int_0 ^t \left(\frac{\mu_0}{2} +\bar{b}(\tau) \right)e^{\frac{\tau}{2}}\Phi_+ (\tau) d\tau  \label{b-in}
\end{align} 
where we put
\begin{align*}
\bar{b}(\tau) =2 a_1 (\tau) + 2 a_2(\frac{\tau}{4}).
\end{align*} 
From Lemma \ref{gr}, we have that
\begin{align*}
\Phi_+ (t) \leq C \Phi_+ (0).
\end{align*}

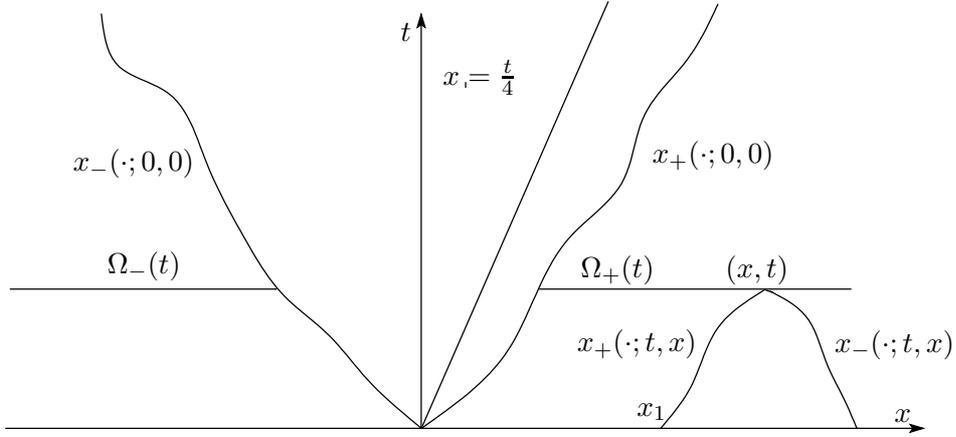
\begin{center}{Figure 3: $\Omega_\pm (t)$, characteristic curves and the line $x=\frac{t}{4}$}
{\unitlength 0.1in%
\begin{picture}(46.8800,24.4000)(3.4000,-24.8000)%
%
\special{pn 8}%
\special{pa 340 2467}%
\special{pa 4754 2467}%
\special{fp}%
%
\special{pn 8}%
\special{pa 2493 2467}%
\special{pa 2547 2429}%
\special{pa 2573 2410}%
\special{pa 2600 2390}%
\special{pa 2626 2371}%
\special{pa 2704 2311}%
\special{pa 2729 2291}%
\special{pa 2753 2270}%
\special{pa 2777 2248}%
\special{pa 2800 2227}%
\special{pa 2845 2182}%
\special{pa 2887 2134}%
\special{pa 2906 2110}%
\special{pa 2925 2084}%
\special{pa 2959 2032}%
\special{pa 2975 2004}%
\special{pa 2990 1976}%
\special{pa 3004 1948}%
\special{pa 3032 1890}%
\special{pa 3045 1860}%
\special{pa 3058 1831}%
\special{pa 3071 1801}%
\special{pa 3085 1772}%
\special{pa 3098 1742}%
\special{pa 3111 1713}%
\special{pa 3125 1683}%
\special{pa 3140 1654}%
\special{pa 3170 1598}%
\special{pa 3187 1570}%
\special{pa 3204 1544}%
\special{pa 3223 1517}%
\special{pa 3242 1492}%
\special{pa 3263 1468}%
\special{pa 3285 1444}%
\special{pa 3308 1421}%
\special{pa 3356 1377}%
\special{pa 3381 1355}%
\special{pa 3405 1333}%
\special{pa 3428 1311}%
\special{pa 3451 1288}%
\special{pa 3473 1265}%
\special{pa 3493 1241}%
\special{pa 3512 1216}%
\special{pa 3529 1190}%
\special{pa 3543 1162}%
\special{pa 3555 1134}%
\special{pa 3566 1104}%
\special{pa 3576 1073}%
\special{pa 3584 1042}%
\special{pa 3592 1010}%
\special{pa 3600 979}%
\special{pa 3608 947}%
\special{pa 3626 885}%
\special{pa 3638 855}%
\special{pa 3651 826}%
\special{pa 3666 799}%
\special{pa 3682 772}%
\special{pa 3720 720}%
\special{pa 3740 695}%
\special{pa 3761 671}%
\special{pa 3783 646}%
\special{pa 3804 622}%
\special{pa 3825 597}%
\special{pa 3846 573}%
\special{pa 3866 547}%
\special{pa 3902 495}%
\special{pa 3919 468}%
\special{pa 3951 412}%
\special{pa 3966 384}%
\special{pa 4008 297}%
\special{pa 4021 268}%
\special{pa 4028 253}%
\special{fp}%
%
\special{pn 8}%
\special{pa 2493 2467}%
\special{pa 2468 2446}%
\special{pa 2444 2426}%
\special{pa 2419 2405}%
\special{pa 2395 2384}%
\special{pa 2370 2363}%
\special{pa 2346 2342}%
\special{pa 2323 2321}%
\special{pa 2299 2299}%
\special{pa 2276 2277}%
\special{pa 2253 2254}%
\special{pa 2231 2231}%
\special{pa 2210 2208}%
\special{pa 2188 2184}%
\special{pa 2168 2159}%
\special{pa 2147 2135}%
\special{pa 2107 2085}%
\special{pa 2086 2061}%
\special{pa 2065 2036}%
\special{pa 2044 2012}%
\special{pa 2022 1989}%
\special{pa 1976 1945}%
\special{pa 1952 1924}%
\special{pa 1927 1903}%
\special{pa 1903 1883}%
\special{pa 1853 1841}%
\special{pa 1829 1820}%
\special{pa 1806 1799}%
\special{pa 1783 1776}%
\special{pa 1762 1753}%
\special{pa 1741 1729}%
\special{pa 1721 1703}%
\special{pa 1702 1678}%
\special{pa 1684 1651}%
\special{pa 1666 1625}%
\special{pa 1649 1597}%
\special{pa 1632 1570}%
\special{pa 1568 1458}%
\special{pa 1553 1430}%
\special{pa 1537 1403}%
\special{pa 1477 1291}%
\special{pa 1463 1263}%
\special{pa 1421 1176}%
\special{pa 1408 1147}%
\special{pa 1395 1117}%
\special{pa 1383 1086}%
\special{pa 1371 1056}%
\special{pa 1359 1024}%
\special{pa 1347 993}%
\special{pa 1336 961}%
\special{pa 1323 930}%
\special{pa 1311 899}%
\special{pa 1297 870}%
\special{pa 1282 841}%
\special{pa 1266 813}%
\special{pa 1249 787}%
\special{pa 1229 762}%
\special{pa 1208 739}%
\special{pa 1184 719}%
\special{pa 1158 701}%
\special{pa 1130 685}%
\special{pa 1070 657}%
\special{pa 1039 644}%
\special{pa 1009 630}%
\special{pa 981 616}%
\special{pa 954 599}%
\special{pa 930 581}%
\special{pa 910 559}%
\special{pa 893 535}%
\special{pa 879 508}%
\special{pa 868 479}%
\special{pa 859 448}%
\special{pa 852 415}%
\special{pa 846 381}%
\special{pa 841 347}%
\special{pa 837 312}%
\special{pa 836 304}%
\special{fp}%
%
\special{pn 8}%
\special{pa 4274 1741}%
\special{pa 4190 1795}%
\special{pa 4164 1814}%
\special{pa 4137 1833}%
\special{pa 4112 1853}%
\special{pa 4089 1874}%
\special{pa 4066 1896}%
\special{pa 4046 1919}%
\special{pa 4027 1944}%
\special{pa 4010 1970}%
\special{pa 3995 1998}%
\special{pa 3981 2026}%
\special{pa 3968 2056}%
\special{pa 3956 2085}%
\special{pa 3945 2116}%
\special{pa 3933 2146}%
\special{pa 3921 2177}%
\special{pa 3909 2206}%
\special{pa 3896 2236}%
\special{pa 3881 2264}%
\special{pa 3865 2292}%
\special{pa 3848 2319}%
\special{pa 3810 2371}%
\special{pa 3791 2396}%
\special{pa 3770 2421}%
\special{pa 3750 2446}%
\special{pa 3731 2468}%
\special{fp}%
%
\special{pn 8}%
\special{pa 4299 1751}%
\special{pa 4328 1767}%
\special{pa 4358 1783}%
\special{pa 4386 1799}%
\special{pa 4413 1817}%
\special{pa 4438 1836}%
\special{pa 4461 1857}%
\special{pa 4482 1880}%
\special{pa 4500 1904}%
\special{pa 4516 1931}%
\special{pa 4530 1960}%
\special{pa 4542 1989}%
\special{pa 4586 2113}%
\special{pa 4598 2144}%
\special{pa 4611 2173}%
\special{pa 4639 2231}%
\special{pa 4654 2259}%
\special{pa 4668 2287}%
\special{pa 4696 2345}%
\special{pa 4709 2374}%
\special{pa 4721 2404}%
\special{pa 4734 2433}%
\special{pa 4746 2463}%
\special{pa 4747 2465}%
\special{fp}%
%
\special{pn 8}%
\special{pa 3105 1740}%
\special{pa 4717 1740}%
\special{fp}%
%
\special{pn 8}%
\special{pa 1737 1740}%
\special{pa 1737 1740}%
\special{fp}%
%
\special{pn 8}%
\special{pa 1737 1740}%
\special{pa 364 1740}%
\special{fp}%
\put(8.6600,-16.9500){\makebox(0,0)[lb]{$\Omega_{-} (t)$}}%
\put(40.6800,-17.1500){\makebox(0,0)[lb]{$(x,t)$}}%
\put(33.1800,-17.1500){\makebox(0,0)[lb]{$\Omega_{+} (t)$}}%
\put(36.9000,-11.2000){\makebox(0,0)[lb]{$x_+ (\cdot;0,0)$}}%
\put(6.9000,-11.6000){\makebox(0,0)[lb]{$x_- (\cdot;0,0)$}}%
%
\special{pn 8}%
\special{pa 2493 2467}%
\special{pa 3460 240}%
\special{fp}%
\put(36.1000,-24.2300){\makebox(0,0)[lb]{$x_1$}}%
\put(46.3200,-21.0400){\makebox(0,0)[lb]{$x_- (\cdot;t,x)$}}%
\put(33.0000,-21.0300){\makebox(0,0)[lb]{$x_+ (\cdot;t,x)$}}%
\put(49.4200,-24.3100){\makebox(0,0)[lb]{$x$}}%
\put(23.8400,-4.4400){\makebox(0,0)[lb]{$t$}}%
\put(26.0500,-7.2500){\makebox(0,0)[lb]{$x=\frac{ t }{4}$}}%
%
\special{pn 8}%
\special{pa 4312 1757}%
\special{pa 4277 1746}%
\special{fp}%
%
\special{pn 8}%
\special{pa 2722 687}%
\special{pa 2722 661}%
\special{fp}%
%
\special{pn 8}%
\special{pa 4751 2467}%
\special{pa 5091 2467}%
\special{fp}%
\special{sh 1}%
\special{pa 5091 2467}%
\special{pa 5024 2447}%
\special{pa 5038 2467}%
\special{pa 5024 2487}%
\special{pa 5091 2467}%
\special{fp}%
%
\special{pn 8}%
\special{pa 2492 2466}%
\special{pa 2492 306}%
\special{fp}%
\special{sh 1}%
\special{pa 2492 306}%
\special{pa 2472 373}%
\special{pa 2492 359}%
\special{pa 2512 373}%
\special{pa 2492 306}%
\special{fp}%
\end{picture}}%
%
\end{center}

\end{proof}

\section{Estimates of $r$ and $s$ in the inner region}
In this section, we show a estimate for $r,s$ in the inner region
\begin{align*}
I(t) =\{ x \in \R  \ | \ x_- (t;0,0) \leq x \leq x_+ (t;0,0)  \}.
\end{align*} 
The following type of the maximal principle is found in \cite{CLLMZ} (see Remark \ref{re-cllmz}).
\begin{Lemma} \label{es-in} Let $\mu_0 \leq 0$. Suppose that \eqref{asg-d} and \eqref{bo1} are satisfied.  Then the following estimates hold 
\begin{align} 
\max \{ \sup_{[0,T]}\|r(t,x) \|_{L^\infty (I(t))}, \sup_{[0,T]}\|s(t,x)\|_{L^\infty (I(t))} \}   \notag \\
\leq  \max \{  \sup_{[0,T]}|r|(t,x_+ (t)),  \sup_{[0,T]}|s|(t,x_+ (t)),   \sup_{[0,T]}|r|(t,x_- (t)),  \sup_{[0,T]}|s|(t,x_- (t)) \},
\label{es-in-con}  
\end{align}
where we abbreviate $x_\pm (t)) = x_\pm (t;0,0))$.
\end{Lemma}
\begin{proof}
For arbitrarily $\delta>0$,  we set 
\begin{align*}
r_{\delta}= r e^{-\delta t}, \\
s_\delta=s e^{-\delta t}.
\end{align*}
$R_\delta$ and $S_\delta$ are solutions to 
\begin{eqnarray}\label{wwd}\left\{
\begin{array}{ll} 
\partial_{-} r_\delta + \delta r_\delta =- \frac{a}{2}(r_\delta + s_\delta) \\
\partial_{-} s_\delta + \delta s_\delta =- \frac{a}{2}(r_\delta + s_\delta)
\end{array}\right.
\end{eqnarray}
First, for $r_\delta$ and $s_\delta$, we show that \eqref{es-in-con} is true.
In the contradiction argument, we suppose that $|r_\delta|$ or $|s_\delta|$ has a maximum value t at $(t_0,x_0)$ in the interior of $\cup_{[0,T]} I(t)$ or $\{  t=T \} \times (x_- (T;0,0), x_+(T;0,0))$.
From the symmetry, we can assume that $|r_\delta (t_0, x_0) | \geq |s_\delta (t_0,x_0)|$.
It is obvious that we can assume that $|r_\delta (t_0, x_0) |>0$ in this contradiction argument.
Thus $|r_\delta|$ is $C^1$ in the neighbor of $(t_0 ,x_0)$.  Multiplying the first equation of \eqref{wwd} by $|r_\delta|^{-1} r_\delta$, we have that
\begin{align*}
|r_\delta|_t - c |r_\delta|_x  +\delta |r_\delta|=- \frac{a}{2}(|r_\delta| +|r_\delta|^{-1} r_\delta s_\delta ).
\end{align*}
From the standard argument, we can easily see that
\begin{align*}
|r_\delta (t_0,x_0)|_t \geq 0 \ \mbox{and} \   |r_\delta (t_0,x_0)|_x =0.
\end{align*}
Hence the left had side is  strictly positive. On the other hand,  from \eqref{asg-d}, it follows that
\begin{align*}
\frac{a}{2}(|r_\delta (t_0, x_0)| +|r|^{-1} r_\delta s_\delta (t_0 ,x_0)) \geq\frac{a}{2}( |r_\delta (t_0, x_0)|  - |s_\delta (t_0 ,x_0)|) \geq 0,
\end{align*}
which leads a contradiction. 
Therefore, we have \eqref{es-in-con} for $r_\delta$ and $s_\delta$, which implies  that
\begin{align*} 
\max \{ \sup_{[0,T]}\|r(t,x) \|_{L^\infty (I(t))}, \sup_{[0,T]}\|s(t,x)\|_{L^\infty (I(t))} \} \\
\leq e^{ \delta T} \max \{  \sup_{[0,T]}|r|(t,x_+ (t)),  \sup_{[0,T]}|s|(t,x_+ (t)),   \sup_{[0,T]}|r|(t,x_- (t)),  \sup_{[0,T]}|s|(t,x_- (t)) \}.
\end{align*}
Taking $\delta \rightarrow 0$, we have \eqref{es-in-con} for $r$ and $s$.
\end{proof}
\begin{Remark} \label{re-cllmz}
In \cite{CLLMZ}, they have shown this type of maximal principle in  $[0,T] \times [a,b]$.
In their  proof, the following energy   functional  is used
\begin{align*}
\int_a ^b |r|^{n} + |s|^{n} dx,
\end{align*}
where $n$ is an even natural number.  Estimating this functional and taking $n \rightarrow \infty$,
they obtain the desired estimate. It would be possible to apply their strategy to the proof of Lemma \ref{es-in}.
However, this paper give a more simple and maximal principle like proof.
\end{Remark}

\section{Proof of Theorem\ref{main5}}
\subsection{Continuity argument}
Combining Lemma \ref{esP} and the continuity argument, we can show the following lemma to control $u, v$ and $c$
\begin{Lemma} \label{s-es} If \eqref{asg-a}-\eqref{asg-c} are satisfied,  then the following estimates hold on $[0, T^*)$ for sufficient small $\varepsilon >0$
with $t\in [0,T^*)$ and  $x \in \Omega (t)$ respectively
\begin{align}
|u(t,x) -1| \leq  C\varepsilon, \label{u-a}  \\
|v(t,x)| \leq C\varepsilon, \label{v-e} \\ 
|r(t,x) | + |s(t,x)| \leq C\varepsilon . \label{rs-e}  
\end{align}
Moreover, if \eqref{asg-c} is also assumed,  then the above three estimates hold for $t \in [0,T^*)$ and $x \in \R$. 
\end{Lemma}
\begin{proof}
We show this lemma by the continuity argument.
We fix $T \in (0,T^*)$ arbitrarily. 
From the continuity of the solution, \eqref{bo1} holds neat $t=0$, if $\varepsilon >0$ is small.
Let us denote by $T _1$ the maximal time that \eqref{bo1} holds.
We suppose that $T_1 < T$ in the contradiction argument.
From Lemma \ref{esP}, \eqref{rs-e} holds on $[0, T_1]$, which implies that \eqref{u-a} and \eqref{v-e} for sufficient small $\varepsilon$.
Thus, for sufficiently small $\varepsilon >0$, we find that it holds on $(t,x) \in [0,T_1] \times \Omega (t)$ 
\begin{align*}
\frac{1}{2} \leq c(u) \leq 2.
\end{align*}
We note that this smallness of $\varepsilon $ is independent of $T$.
Hence, from the continuity of the solution, the estimate \eqref{bo1} can be extend to the interval $[0,T_2]$ for $T_1 < T_2 \leq T$,
which contradicts to the definition of $T_1$.
From the arbitrariness of $T$, we see that \eqref{bo1} holds on $[0,T^*)$, which implied \eqref{u-a}, \eqref{v-e} and \eqref{rs-e} hold on $[0,T^*)$ by Lemma \ref{esP}.

Furthermore,under the assumption \eqref{asg-d}, one can find that the above argument is valid in $\R$, 
since the right hand side of \eqref{es-in-con} in Lemma \eqref{es-in} is bounded by $C\varepsilon $. 
\end{proof}
\subsection{Conclusion}
First we consider the case that $x_0 \geq 0$. We consider the equation of $Q$ in \eqref{eq2x} on the characteristic curve $x_+ (\cdot;0,x_0)$ .

It should be noted that the following uniform estimate holds
\begin{eqnarray} \label{a-es-c}
\int_0 ^t | a(\tau, x_{\pm}(\tau;0,0))| + | a_t(\tau, x_{\pm}(\tau;0,0))| +  |a_x(\tau, x_{\pm}(\tau;0,0))| d\tau \leq C. 
\end{eqnarray}
We also remark that  it holds that $x_+ (t;0,x_0) \in \Omega_+ (t) $.
Applying this inequality, \eqref{u-a}, \eqref{v-e}, \eqref{rs-e} and Lemma \ref{esA} to \eqref{eq2x} (note that $Y$ and $R$ do not appear in \eqref{eq2x}), we have that
\begin{eqnarray} \label{q-ineq}
Q(t)\leq Q(0) + K^*  \varepsilon - C\int_0 ^t  Q(\tau)^2 d\tau,
\end{eqnarray}
where $K^*$ is a positive constant depending on $\|\phi \|_{L^\infty}$, $\|\psi \|_{L^\infty}$, $\|\phi_x \|_{L^\infty}$ and   $C_a$ \eqref{asg-a}.
From the assumption  \eqref{KK}, it follows that
\begin{align} \label{q-q}
	Q(t)\leq  -C_1\varepsilon(K-K^*) - C_2\int_0 ^t  Q(\tau)^2 d\tau,
\end{align}
which implies that $-Q(t)$ blows up in finite time for sufficiently large $K.$
In fact, putting $q(t)$ as  a solution of the integral equation:
\begin{align*}
q(t)=  C_1\varepsilon(K-K^*) +C_2\int_0 ^t  q(\tau)^2 d\tau,
\end{align*}
we can easily check that $-Q(t) \geq q(t)$. Solving this integral equation, we have that
\begin{align*}
q(t) = \frac{C_1\varepsilon(K-K^*)}{1- C_1C_2\varepsilon(K-K^*)t}
\end{align*}
Thus  we also obtain the upper estimate of the life-span for $Q(t)$
\begin{align*}
T^* \leq \frac{1}{C_1C_2\varepsilon(K-K^*)}.
\end{align*}
In the case that $x_0 <0$, we can show that  $Y(t)$ in \eqref{eq1x} on the characteristic curve  $x_- (\cdot;0,x_0)$  blows up in finite time.

Next we show the estimate of the  life-span  from below.
In the same way as in the proof of \eqref{q-q},if \eqref{asg-d} is assumed, from Lemma \ref{s-es}, we can obtain that 
\begin{align*}
	|Q(t)|\leq  C_3+ C_4\int_0 ^t  Q(\tau)^2 d\tau.
\end{align*}
Similarly, we have that on the characteristic curve $x_- (\cdot;t,x)$
\begin{align*}
	|Y(t)|\leq  C_3 + C_4\int_0 ^t  Y(\tau)^2 d\tau.
\end{align*}
Solving these integral equations, we have that
\begin{align*}
T^* \geq C\varepsilon^{-1}.
\end{align*} 
From the boundedness of $u$, $v$, $p'$, time or space derivatives of $(u,v)$ blow up at $T^*$.
From the equations in \eqref{de0}, the blow-up of $(u_t, v_t)$ implies that of $(u_x ,v_x)$.
Hence we complete the proof of Theorem \ref{main5}.
\section{Proof of Theorem \ref{main5g}}
\subsection{Estimate of $r_x$ and $s_x$}
Since the estimates of the functions $r$ and $s$ itself has already been completed in Lemma \ref{s-es}, all that remains is that of derivatives.
We show the boundedness of $r_x$ and $s_x$ as follows. This lemma immediately implies the completeness of the proof of Theorem \ref{main5g}. 
\begin{Lemma} \label{s-esx} Let $\mu_0 >0$. The following estimates hold on $[0, T^*)$ for sufficient small $\varepsilon= \varepsilon(\|f\|_{C^1 _b} , \|g\|_{C^1 _b})>0$
\begin{align}
|r_x(t,x) | + |s_x(t,x)| \leq C\varepsilon . \label{rsx-e}  
\end{align}
\end{Lemma}
\begin{proof}
Applying \eqref{a-es-c}, \eqref{u-a}, \eqref{v-e}, \eqref{rs-e} and Lemma \ref{esA} to \eqref{eq1x} and \eqref{eq2x}, we have that
\begin{eqnarray} \label{qq-ineq}
\|r_x (t) \|_{L^\infty}  \leq C_0 \varepsilon +  e^{-\frac{\mu_0 t}{2}} \int_0 ^t e^{\frac{\mu_0 \tau}{2}} \| r_x (\tau)\|^2_{L^\infty} d\tau
\end{eqnarray}
and
\begin{eqnarray} \label{y-ineq}
\|s_x (t) \|_{L^\infty}  \leq C_0 \varepsilon +  e^{-\frac{\mu_0 t}{2}} \int_0 ^t e^{\frac{\mu_0 \tau}{2}} \| s_x (\tau)\|^2_{L^\infty} d\tau,
\end{eqnarray}
which immediately implies \eqref{rsx-e}, if $\varepsilon >0$ is sufficiently small.

\end{proof}

\section*{Acknowledgments}
The author would like to Thank Prof. Hiroyuki Takamura for his comment on generalization of the damping coefficient.
The research is  supported by  Grant-in-Aid for Early-Career Scientists, No. 19K14573.





\end{document}